\documentclass[11pt,a4paper]{article}
\usepackage[margin=25mm]{geometry}
\usepackage{amsmath,amssymb,amsthm,booktabs,array,longtable}
\usepackage[T1]{fontenc}
\usepackage{mathptmx} 
\usepackage[scaled=0.92]{helvet}
\usepackage{courier}
\usepackage[hidelinks]{hyperref}
\newtheorem{theorem}{Theorem}[section]
\newtheorem{lemma}[theorem]{Lemma}
\newtheorem{definition}[theorem]{Definition}
\newcommand{\Ric}{\operatorname{Ric}}
\newcommand{\Der}{\operatorname{Der}}
\newcommand{\Id}{\operatorname{Id}}
\newcommand{\diag}{\operatorname{diag}}
\newcommand{\spn}{\operatorname{span}}
\newcommand{\nB}{\nabla^{B}}
\newcommand{\nL}{\nabla^{L}}
\newcommand{\RB}{R^{B}}
\newcommand{\LB}{\mathcal{L}^{B}}
\title{Algebraic and Symmetrized Affine $\ast$-Ricci Solitons for the Bott Connection on Three-Dimensional Lorentzian Lie Groups}
\author{YUXIN DING}\date{}
\begin{document}\maketitle
\begin{abstract}
We fix a spacelike two-dimensional left-invariant distribution and its associated almost contact structure on a three-dimensional Lorentzian Lie group, and investigate the star-Ricci contraction of the curvature of the Bott connection. We compute the complete tensor for seven standard bracket families and determine all real parameter solutions of the general algebraic equation and its self-adjoint subclass. Using the symmetric part of the tensor, we also define a Bott affine star-Ricci soliton equation and classify all left-invariant vector field solutions. Every affine solution is steady, while two non-unimodular families admit branches with a nonzero symmetrized tensor. The proofs provide complete component equations, treat all degenerate parameter branches, and establish sufficiency by explicit substitution. We also distinguish almost contact structures from contact metric structures.
\end{abstract}
\noindent\textbf{Keywords:} Lorentzian Lie group; Bott connection; star-Ricci tensor; derivation; affine soliton.

\section{Introduction}
Soliton problems in three-dimensional left-invariant Lorentzian geometry can be reduced, through the structure constants, to finite-dimensional algebraic systems. Curvature contractions depend on the connection under consideration, however, so a classification of solitons for the Levi--Civita connection cannot be transferred directly to the Bott connection. Wu and Wang studied the Bott connection and associated affine Ricci solitons on three-dimensional Lorentzian Lie groups \cite{WuWang}. Zhu and Ma investigated algebraic star-Ricci solitons on three-dimensional Lorentzian almost contact Lie groups \cite{ZhuMa}. In the present paper, we fix a spacelike two-dimensional distribution and its associated almost contact structure, and study the star-Ricci contraction of Bott curvature.

The Bott connection is generally not metric-compatible, and its star-Ricci tensor need not be symmetric. Consequently, two steps that are sometimes conflated must be treated separately. A nondegenerate metric defines the metric-dual endomorphism of any bilinear tensor without requiring symmetry. By contrast, when a symmetric soliton equation is constructed using an affine metric deformation term, the symmetric part of the curvature tensor is the natural quantity to use. We therefore study three distinct problems: the general algebraic equation, its self-adjoint algebraic subclass, and the Bott affine equation defined by the symmetrized tensor. All definitions are stated explicitly below, independently of the conventions adopted in other references.

Our main results are as follows. First, the complete star-Ricci tensors for the seven standard bracket families can be expressed uniformly through three scalars $K,U,W$, and the nonsymmetric part is determined entirely by $U,W$. Second, we determine every real parameter branch satisfying $Q=c\Id+D$ and identify the branches excluded by self-adjointness. Third, we classify all left-invariant vector field solutions of the symmetrized Bott affine equation. Every affine solution is steady, although $G_6,G_7$ admit solutions with nonzero symmetrized star-Ricci tensors. The proofs give the complete component equations, resolve all zero-parameter cases, and verify every surviving branch by substitution.

Throughout, $H=\spn\{e_1,e_2\}$ is fixed; the choice of distribution is not part of the classification. The seven labels refer to standard parameterized bracket families, not to only seven abstract Lie groups. The parameter lists are not moduli spaces obtained by a further quotient by all metric isomorphisms or scaling equivalences. All affine solitons considered here have left-invariant vector fields.

\section{Basic Definitions and Distinct Soliton Problems}
\subsection{Precise scope of the investigation}
Fix a pseudo-orthonormal basis and the almost contact structure
\[
g=\diag(1,1,-1),\quad\phi e_1=e_2,\quad\phi e_2=-e_1,\quad\phi e_3=0,\quad\xi=e_3,\quad\eta=-g(\,\cdot\,,e_3).
\]
We consider left-invariant structures on connected, simply connected Lie groups, with the above $\phi$ and distribution fixed. The seven families are standard bracket parameter families, not seven abstract groups. The results do not classify all possible distributions or almost contact structures. Write $\varepsilon=\pm1$ for the discrete parameter of $G_4$, to distinguish it from the one-form $\eta$. We use $D$ for a derivation and $H$ for the distribution.

The curvature convention is always
\[
\RB(X,Y)Z=\nB_X\nB_YZ-\nB_Y\nB_XZ-\nB_{[X,Y]}Z.
\]
Let $\epsilon_1=\epsilon_2=1$, $\epsilon_3=-1$ denote the metric signs. We define
\begin{equation}
T(X,Y):=\Ric^*(X,Y)=\frac12\sum_{i=1}^{3}\epsilon_i g\bigl(\RB(X,\phi Y)\phi e_i,e_i\bigr).
\label{eq:star}
\end{equation}
This is a signed metric contraction. Since $\phi e_3=0$ in the present adapted basis, only the first two terms contribute. The metric signs must nevertheless be retained for a general pseudo-orthonormal basis.

\begin{definition}
Let $T=\Ric^*$ be defined by the Bott curvature contraction above, and let $Q$ be the unique endomorphism satisfying $g(QX,Y)=T(X,Y)$. The fixed structure satisfies the \emph{general algebraic star-Ricci soliton equation} if there exist $c\in\mathbb R$ and $D\in\Der(\mathfrak g)$ such that
\[
Q=c\Id+D,\qquad D[X,Y]=[DX,Y]+[X,DY].
\]
If $T(X,Y)=T(Y,X)$ also holds, the solution belongs to the \emph{self-adjoint algebraic subclass}.
\end{definition}
Self-adjointness is with respect to the Lorentzian metric and cannot be replaced by ordinary matrix symmetry. The self-adjoint subclass and general class are stated separately below to make their scopes explicit.

\begin{definition}
Set $S=(T+T^{\mathsf T})/2$, where $T^{\mathsf T}(X,Y)=T(Y,X)$. A quadruple $(g,\phi,V,\lambda)$ is a \emph{symmetrized Bott affine star-Ricci soliton} if $V$ is left-invariant, $\lambda\in\mathbb R$, and
\[
g(\nB_XV,Y)+g(X,\nB_YV)+2S(X,Y)=2\lambda g(X,Y).
\]
\end{definition}
Both algebraic endomorphisms are defined from the original tensor $T$, whereas the affine equation uses $S$. Substituting the metric dual of $S$ into the algebraic equation would produce a different problem, which we do not use in place of the definitions above. We call a general algebraic solution expanding, steady, or shrinking according as $c<0$, $c=0$, or $c>0$. The affine sign convention causes no ambiguity, since all affine solutions have $\lambda=0$.

\section{The Seven Parameter Families and a Unified Computation}
\subsection{Lie brackets and parameter restrictions}
All parameters are real, and undisplayed brackets are determined by skew-symmetry.
\begin{align*}
G_1:\quad &[e_1,e_2]=\alpha e_1-\beta e_3,\quad [e_1,e_3]=-\alpha e_1-\beta e_2,\\
&[e_2,e_3]=\beta e_1+\alpha e_2+\alpha e_3,\qquad\alpha\ne0;\\[3pt]
G_2:\quad &[e_1,e_2]=\gamma e_2-\beta e_3,\quad[e_1,e_3]=-\beta e_2-\gamma e_3,\\
&[e_2,e_3]=\alpha e_1,\qquad\gamma\ne0;\\[3pt]
G_3:\quad &[e_1,e_2]=-\gamma e_3,\quad[e_1,e_3]=-\beta e_2,\quad[e_2,e_3]=\alpha e_1;\\[3pt]
G_4:\quad &[e_1,e_2]=-e_2+(2\varepsilon-\beta)e_3,\quad[e_1,e_3]=-\beta e_2+e_3,\\
&[e_2,e_3]=\alpha e_1,\qquad\varepsilon=\pm1;\\[3pt]
G_5:\quad &[e_1,e_2]=0,\quad[e_1,e_3]=\alpha e_1+\beta e_2,\quad[e_2,e_3]=\gamma e_1+\delta e_2,\\
&\alpha+\delta\ne0,\qquad\alpha\gamma+\beta\delta=0;\\[3pt]
G_6:\quad &[e_1,e_2]=\alpha e_2+\beta e_3,\quad[e_1,e_3]=\gamma e_2+\delta e_3,\quad[e_2,e_3]=0,\\
&\alpha+\delta\ne0,\qquad\alpha\gamma-\beta\delta=0;\\[3pt]
G_7:\quad &[e_1,e_2]=-\alpha e_1-\beta e_2-\beta e_3,\quad[e_1,e_3]=\alpha e_1+\beta e_2+\beta e_3,\\
&[e_2,e_3]=\gamma e_1+\delta e_2+\delta e_3,\qquad\alpha+\delta\ne0,\quad\alpha\gamma=0.
\end{align*}
\subsection{Generating the connection and curvature directly from the brackets}
Write $[e_i,e_j]=\sum_k C_{ij}^{k}e_k$ and let the $j$th column of $A_i$ represent $\nB_{e_i}e_j$. The Koszul formula reads
\[
2g(\nL_{e_i}e_j,e_k)=C_{ij}^{k}\epsilon_k-C_{jk}^{i}\epsilon_i+C_{ki}^{j}\epsilon_j.
\]
With $\mathcal V=H^\perp$, the four projection rules defining the Bott connection are
\[
\nB_XY=\begin{cases}
\pi_H(\nL_XY),&X,Y\in\Gamma(H),\\
\pi_H([X,Y]),&X\in\Gamma(\mathcal V),\ Y\in\Gamma(H),\\
\pi_{\mathcal V}([X,Y]),&X\in\Gamma(H),\ Y\in\Gamma(\mathcal V),\\
\pi_{\mathcal V}(\nL_XY),&X,Y\in\Gamma(\mathcal V).
\end{cases}
\]
Then
\begin{equation}
R_{ij}=A_iA_j-A_jA_i-\sum_k C_{ij}^{k}A_k.
\label{eq:curvature}
\end{equation}
This avoids treating a quoted curvature table as an unchecked premise. The appendix lists every $A_i$ and the curvature data needed to reproduce the contraction.

\begin{lemma}
Set
\[
\tau_{ij}=\tfrac12\bigl((R_{ij})_{12}-(R_{ij})_{21}\bigr),\qquad
K=\tau_{12},\quad U=-\tau_{23},\quad W=\tau_{13}.
\]
For all seven families, the star-Ricci tensor and its metric-dual endomorphism are
\begin{equation}
[T_{ij}]=\begin{pmatrix}K&0&0\\0&K&0\\U&W&0\end{pmatrix},\qquad
[Q]=\begin{pmatrix}K&0&U\\0&K&W\\0&0&0\end{pmatrix}.
\label{eq:unified}
\end{equation}
Columns of the endomorphism matrix represent images of basis vectors, consistent with $g(QX,Y)=T(X,Y)$. In particular, $T$ is symmetric exactly when $U=W=0$.
\end{lemma}
\begin{proof}
In this basis, \eqref{eq:star} is
\[
\tfrac12\{g(\RB(X,\phi Y)e_2,e_1)-g(\RB(X,\phi Y)e_1,e_2)\}.
\]
Taking $Y=e_1,e_2,e_3$ yields the three tensor columns. If $E=\diag(1,1,-1)$, then $T=Q^{\mathsf T}E$, so $Q=ET^{\mathsf T}$, giving \eqref{eq:unified}.
\end{proof}
\begin{table}[htbp]\centering
\caption{Complete star-Ricci data. The constraint $\alpha\gamma=0$ is used for $G_7$.}
\begin{tabular}{@{}c c c c@{}}\toprule Family&$K$&$U$&$W$\\\midrule
$G_1$&$-\alpha^2-\beta^2$&$0$&$\alpha^2$\\
$G_2$&$-\gamma^2-\tfrac12\beta(\alpha+\beta)$&$0$&$\tfrac12\gamma(\beta-\alpha)$\\
$G_3$&$-\tfrac12\gamma(\alpha+\beta)$&$0$&$0$\\
$G_4$&$\tfrac12(\alpha+\beta)(2\varepsilon-\beta)-1$&$0$&$\tfrac12(\alpha-\beta)$\\
$G_5$&$0$&$0$&$0$\\
$G_6$&$-\alpha^2-\tfrac12\beta\gamma$&$0$&$-\tfrac12\gamma(2\alpha+\delta)$\\
$G_7$&$-\alpha^2-\tfrac12\beta(\beta+\gamma)$&$-\tfrac12\delta(\beta+\gamma)$&$-K$\\\bottomrule
\end{tabular}\end{table}

\section{Complete Classification of the Self Adjoint Algebraic Subclass}
\begin{theorem}\label{thm:self}
Require $T$ to be symmetric and $Q=c\Id+D$, with $D\in\Der(\mathfrak g)$. The following list is complete. Each item remains subject to the original parameter restrictions; no other cases admit solutions.
\begin{enumerate}
\item $G_1$: no solutions.
\item $G_2$: $\alpha=\beta=0$, $\gamma\ne0$; $c=-\gamma^2$, $D=\diag(0,0,\gamma^2)$.
\item $G_3$: $\gamma(\alpha+\beta)=0$. In the nonabelian case, $c=0$, $D=0$. In the abelian case $\alpha=\beta=\gamma=0$, any $c\in\mathbb R$ is allowed, with $D=-c\Id$.
\item $G_4$: $\alpha=\beta=\varepsilon$; $c=0$, $D=0$.
\item $G_5$: all admissible parameters; $c=0$, $D=0$.
\item $G_6$: precisely two disjoint branches:
\begin{enumerate}
\item $\beta=\gamma=0$, $\alpha+\delta\ne0$; $c=-\alpha^2$, $D=\diag(0,0,\alpha^2)$;
\item $\alpha\ne0$, $\delta=-2\alpha$, $\gamma=2s\alpha$, $\beta=-s\alpha$, $s=\pm1$; $c=0$, $D=0$.
\end{enumerate}
\item $G_7$: $\alpha=0$, $\gamma=-\beta$, $\delta\ne0$; $c=0$, $D=0$.
\end{enumerate}
\end{theorem}
Under our convention, $c<0$, $c=0$, and $c>0$ mean expanding, steady, and shrinking, respectively. Except for the nonunique constant in abelian $G_3$, only expanding or steady solutions occur.

\section{Complete Classification of the General Algebraic Class}
Nondegeneracy of the metric ensures that every bilinear tensor has a unique metric-dual endomorphism $Q$ satisfying $g(QX,Y)=T(X,Y)$. Symmetry is the condition for self-adjointness, not for the existence of $Q$. Here we use the dual of the original $T$, without requiring self-adjointness or symmetrizing the tensor.
\begin{theorem}\label{thm:general}
For the possibly non-self-adjoint endomorphism in \eqref{eq:unified}, the classification of $G_2,G_3,G_5,G_6$ is identical to Theorem \ref{thm:self}. The other families are as follows.
\begin{enumerate}
\item $G_1$: $\beta=0$, $c=-\alpha^2$, and
\[D=\begin{pmatrix}0&0&0\\0&0&\alpha^2\\0&0&\alpha^2\end{pmatrix}.\]
\item $G_4$: exactly three branches:
\[\begin{array}{c|c|c|c}(\alpha,\beta)&K&W&c\\\hline
(\varepsilon,\varepsilon)&0&0&0\\(0,\varepsilon)&-1/2&-\varepsilon/2&-1\\(0,2\varepsilon)&-1&-\varepsilon&-1\end{array}\]
Each has $De_1=(K-c)e_1$, $De_2=(K-c)e_2$, $De_3=We_2-ce_3$.
\item $G_7$: the zero-endomorphism branch in Theorem \ref{thm:self}, and the additional branch $\alpha\ne0$, $\beta=\gamma=0$, $\alpha+\delta\ne0$, $c=-\alpha^2$, with the same matrix $D$ as in the $G_1$ item above.
\end{enumerate}
\end{theorem}

\section{Step by Step Elimination and All Degenerate Algebraic Branches}
We first solve the general endomorphism equation, then select the self-adjoint subclass using $U=W=0$. All parameters are real. Before dividing by a parameter we distinguish whether it vanishes, so degenerate Lie algebras are also covered. Define the derivation defects
\[F_{ij}=D[e_i,e_j]-[De_i,e_j]-[e_i,De_j],\qquad i<j.\]
Left-invariance and bilinearity make $F_{12}=F_{13}=F_{23}=0$ sufficient; arbitrary vectors need not be checked separately. The equation $Q=c\Id+D$ gives the unique candidate
\[D=\begin{pmatrix}K-c&0&U\\0&K-c&W\\0&0&-c\end{pmatrix}.\]
Thus there is no need to introduce nine independent derivation coefficients.

\subsection{The family G1 and the nonsymmetric component}
Here $K=-\alpha^2-\beta^2$, $U=0$, $W=\alpha^2$, $\alpha\ne0$. Setting all three defects to zero and removing duplicates gives
\[\alpha(c-K)=0,\quad\beta W=0,\quad\beta(c-2K)=0,\quad\alpha(W+c)=0,\quad\beta c=0.\]
\textbf{Step 1. Use necessarily nonzero quantities.} Since $\alpha\ne0$, $c=K$. Since $W=\alpha^2>0$, $\beta W=0$ forces $\beta=0$.

\textbf{Step 2. Determine the constant.} Now $K=-\alpha^2$, $W=\alpha^2$, and $c=-\alpha^2$. The fourth equation becomes $\alpha(\alpha^2-\alpha^2)=0$, and the remaining equations containing $\beta$ hold automatically.

\textbf{Step 3. Verify the derivation directly.} We have
\[De_1=De_2=0,\qquad De_3=\alpha^2(e_2+e_3).\]
The brackets are $[e_1,e_2]=\alpha e_1$, $[e_1,e_3]=-\alpha e_1$, $[e_2,e_3]=\alpha(e_2+e_3)$. Both sides of the first derivation identity vanish. For the second, the right-hand side is $\alpha^2([e_1,e_2]+[e_1,e_3])=0$. For the third, both sides equal $\alpha^3(e_2+e_3)$. This proves sufficiency independently of elimination.

\textbf{Self-adjointness.} As $W=\alpha^2$ never vanishes, the general class has these solutions but the self-adjoint subclass has none.

\subsection{The family G2 and sums of squares}
Here $\gamma\ne0$, $K=-\gamma^2-\beta(\alpha+\beta)/2$, $W=\gamma(\beta-\alpha)/2$. The equations are
\[\gamma(c-K)-\beta W=0,\quad\beta(c-2K)=0,\quad2\gamma W+\beta c=0,\quad\alpha c=0.\]
\textbf{Step 1. Eliminate $c$.} Since $\gamma\ne0$,
\[c=K+\frac{\beta W}{\gamma}=-\gamma^2-\frac{\alpha\beta+\beta^2}{2}+\frac{\beta^2-\alpha\beta}{2}=-\gamma^2-\alpha\beta.\]
\textbf{Step 2. Determine $\alpha$.} The third equation becomes
\[2\gamma W+\beta c=\gamma^2(\beta-\alpha)+\beta(-\gamma^2-\alpha\beta)=-\alpha(\gamma^2+\beta^2)=0.\]
Since $\gamma^2+\beta^2>0$, $\alpha=0$; this resolves the parameter itself rather than leaving an unresolved product.

\textbf{Step 3. Determine $\beta$.} Now $c=-\gamma^2$, $2K=-2\gamma^2-\beta^2$, so
\[\beta(c-2K)=\beta(\gamma^2+\beta^2)=0\quad\Longrightarrow\quad\beta=0.\]
\textbf{Substitution.} Thus $K=-\gamma^2$, $W=0$, $c=K$, $D=\diag(0,0,\gamma^2)$, and $\alpha c=0$ is automatic. The brackets reduce to $[e_1,e_2]=\gamma e_2$, $[e_1,e_3]=-\gamma e_3$, $[e_2,e_3]=0$. Both sides of the second derivation identity equal $-\gamma^3e_3$; both sides of the others vanish. These solutions are automatically self-adjoint, with no additional non-self-adjoint branches.

\subsection{The family G3 and the separate abelian case}
Here $U=W=0$, $K=-\gamma(\alpha+\beta)/2$, and the equations are exactly
\[\gamma(c-2K)=0,\qquad\alpha c=0,\qquad\beta c=0.\]
\textbf{Branch 1: $(\alpha,\beta)\ne(0,0)$.} One of the last two coefficients is nonzero, hence $c=0$. The first equation becomes $-2\gamma K=\gamma^2(\alpha+\beta)=0$. If $\gamma=0$, it holds automatically and $K=0$. If $\gamma\ne0$, divide by $\gamma^2$ to obtain $\alpha+\beta=0$, again with $K=0$. The combined condition is $\gamma(\alpha+\beta)=0$.

\textbf{Branch 2: $\alpha=\beta=0$, $\gamma\ne0$.} Then $K=0$, and $\gamma c=0$ gives $c=0$. This nonabelian Heisenberg-type bracket must not be included in the abelian case.

\textbf{Branch 3: $\alpha=\beta=\gamma=0$.} All brackets vanish and every linear map is a derivation. Any $c\in\mathbb R$, with $D=-c\Id$, satisfies $Q=0=c\Id+D$. All three branches are self-adjoint; in the nonabelian branches the zero map $D=0$ is a derivation.

\subsection{The family G4 and exactly three general solutions}
Put $r=2\varepsilon-\beta$, $\varepsilon^2=1$. Then $K=(\alpha+\beta)r/2-1$, $W=(\alpha-\beta)/2$, and
\[c=K+rW,\quad r(c-2K)=0,\quad2W-\beta c=0,\quad\alpha c=0.\]
\textbf{Step 1. Eliminate $c$ without dividing by $r$.}
\[c=\frac{(\alpha+\beta)r}{2}-1+\frac{(\alpha-\beta)r}{2}=\alpha r-1.\]
\textbf{Step 2. Simplify the third equation.}
\begin{align*}2W-\beta c&=\alpha-\beta-\beta\{\alpha(2\varepsilon-\beta)-1\}\\&=\alpha(1-2\varepsilon\beta+\beta^2)=\alpha(\beta-\varepsilon)^2.\end{align*}
Hence $\alpha=0$ or $\beta=\varepsilon$.

\textbf{Branch 1: $\alpha\ne0$.} Then $\beta=\varepsilon$, and $\alpha c=0$ gives $c=0$. Thus $0=\alpha\varepsilon-1$, so $\alpha=\varepsilon$. We get $K=W=0$, and every equation holds.

\textbf{Branch 2: $\alpha=0$.} Now $c=-1$ and
\[c-2K=-1-\{\beta(2\varepsilon-\beta)-2\}=(\beta-\varepsilon)^2.\]
The second equation becomes $(2\varepsilon-\beta)(\beta-\varepsilon)^2=0$. Its only possibilities are $\beta=\varepsilon$ or $\beta=2\varepsilon$, giving $(K,W,c)=(-1/2,-\varepsilon/2,-1)$ or $(-1,-\varepsilon,-1)$, respectively. Both triples satisfy the first three equations by direct substitution; the fourth holds because $\alpha=0$.

\textbf{Self-adjointness.} Both triples in Branch 2 have nonzero $W$, so only $\alpha=\beta=\varepsilon$ remains. Initial division by $2\varepsilon-\beta$ would incorrectly omit the general solution $\beta=2\varepsilon$.

\subsection{The family G5 and the restriction on c}
Here $K=U=W=0$, so $D=-c\Id$. The derivation identity on any bracket is
\[-c[X,Y]=-2c[X,Y]\quad\Longleftrightarrow\quad c[X,Y]=0.\]
Since $\alpha+\delta\ne0$, at least one of $\alpha,\delta$ is nonzero and the algebra is nonabelian. The equations $c\alpha=c\beta=c\gamma=c\delta=0$ force $c=0$. Conversely, the zero derivation works. This differs from abelian $G_3$, where $c$ is arbitrary although $Q=0$ in both cases.

\subsection{The family G6 and the trace and parameter constraints}
Here
\[K=-\alpha^2-\frac{\beta\gamma}{2},\quad U=0,\quad W=-\frac{\gamma(2\alpha+\delta)}2,\quad\alpha+\delta\ne0,\quad\alpha\gamma=\beta\delta.\]
The four independent equations are
\begin{align*}\alpha(c-K)+\beta W&=0,&\beta(c-2K)&=0,\\(\delta-\alpha)W+\gamma c&=0,&\delta(c-K)-\beta W&=0.\end{align*}
\textbf{Step 1. Use the trace condition.} Adding the first and fourth equations gives $(\alpha+\delta)(c-K)=0$, hence $c=K$, leaving
\[\beta W=0,\qquad\beta K=0,\qquad(\delta-\alpha)W+\gamma K=0.\]
\textbf{Branch 1: $\beta=0$.} The original constraint is $\alpha\gamma=0$. If $\alpha\ne0$, $\gamma=0$, so $W=0$, $K=-\alpha^2$, and the remaining equation is automatic. If $\alpha=0$, $\delta\ne0$, $K=0$, $W=-\gamma\delta/2$, and the final equation gives $\delta W=-\gamma\delta^2/2=0$, again forcing $\gamma=0$. The branch is
\[\beta=\gamma=0,\quad\alpha+\delta\ne0,\quad c=-\alpha^2,\quad D=\diag(0,0,\alpha^2).\]
\textbf{Branch 2: $\beta\ne0$.} The first two equations force $W=K=0$. If $\gamma=0$, then $K=-\alpha^2=0$ gives $\alpha=0$, and $\beta\delta=0$ gives $\delta=0$, contradicting the trace restriction. Hence $\gamma\ne0$. From $W=0$, $\delta=-2\alpha$; the trace condition becomes $-\alpha\ne0$. The original constraint yields
\[\alpha\gamma=\beta(-2\alpha)\quad\Longrightarrow\quad\gamma=-2\beta.\]
Substitution into $K=0$ gives $0=-\alpha^2+\beta^2$, hence $\beta=\pm\alpha$, equivalently $\beta=-s\alpha$, $\gamma=2s\alpha$, $s=\pm1$. Here $c=K=0$, $D=0$, so sufficiency is immediate. In Branch 1, both sides of the identity for $[e_1,e_2]=\alpha e_2$ vanish; for $[e_1,e_3]=\delta e_3$ they equal $\delta\alpha^2e_3$. Both branches have $W=0$, so the general and self-adjoint classes coincide.

\subsection{The family G7 and separation of the alpha branches}
Here $W=-K$, $K=-\alpha^2-\beta(\beta+\gamma)/2$, $U=-\delta(\beta+\gamma)/2$. The last formula already uses $\alpha\gamma=0$. The independent equations are
\begin{align*}\alpha(K-c)-\beta U&=0,&\beta(2K-c)&=0,\\(\delta-\alpha)U+\gamma c&=0,&\delta(c-K)-\beta U&=0.\end{align*}
\textbf{Step 1. Eliminate $c$.} Subtracting the fourth equation from the first gives $(\alpha+\delta)(K-c)=0$, hence $c=K$. The remaining equations are
\[\beta U=0,\qquad\beta K=0,\qquad(\delta-\alpha)U+\gamma K=0.\]
\textbf{Branch 1: $\alpha=0$.} Then $\delta\ne0$ and $\beta U=-\delta\beta(\beta+\gamma)/2=\delta K$. Thus $\beta U=0$ gives $K=0$, and the final equation yields $\delta U=0$, hence $U=0$. Therefore $\beta+\gamma=0$, whether or not $\beta$ vanishes. We obtain $K=U=W=c=0$, and the zero derivation works.

\textbf{Branch 2: $\alpha\ne0$.} The original constraint gives $\gamma=0$, so $K=-\alpha^2-\beta^2/2<0$. The condition $\beta K=0$ forces $\beta=0$. Thus $U=0$, $K=c=-\alpha^2$, $W=\alpha^2$, and the final equation vanishes. The only restriction on $\delta$ is $\alpha+\delta\ne0$. Again $De_1=De_2=0$, $De_3=\alpha^2(e_2+e_3)$. The brackets become $[e_1,e_2]=-\alpha e_1$, $[e_1,e_3]=\alpha e_1$, $[e_2,e_3]=\delta(e_2+e_3)$. The common values of the two sides of the derivation identity are, respectively, $0$, $0$, $\delta\alpha^2(e_2+e_3)$.

\textbf{Self-adjointness.} Branch 2 has $W=\alpha^2\ne0$, so only Branch 1 remains. Every possible zero-parameter branch has now been examined, and every retained branch has been verified by substitution.

\section{Left Invariant Classification of Symmetrized Bott Affine Star Ricci Solitons}
We use the symmetrized tensor
\[
S(X,Y)=\widetilde{\Ric}^{\,*}(X,Y)=\tfrac12\{T(X,Y)+T(Y,X)\},\qquad
[S]=\begin{pmatrix}K&0&U/2\\0&K&W/2\\U/2&W/2&0\end{pmatrix},
\]
and investigate the Bott affine star-Ricci soliton equation
\begin{equation}
g(\nB_XV,Y)+g(X,\nB_YV)+2\widetilde{\Ric}^{\,*}(X,Y)=2\lambda g(X,Y).
\label{eq:affine}
\end{equation}
Here $V=v_1e_1+v_2e_2+v_3e_3$ is left-invariant and $v_i,\lambda$ are real constants. We may also write $\LB_Vg+2S=2\lambda g$, where $\LB$ denotes the Bott affine expression defined above, rather than the ordinary Lie derivative. If the equation is written as $\LB_Vg+2S+2\mu g=0$, then $\mu=-\lambda$.

The original tensor $T$ is not required to be symmetric in this section. The preceding algebraic classifications use the original tensor and have not been replaced by an algebraic soliton problem for $S$. The symmetrized affine problem is a separate third problem.

\begin{theorem}\label{thm:affine}
All left-invariant solutions of \eqref{eq:affine} are listed below, and every solution has $\lambda=0$. Each item retains the original parameter restrictions. Cases not listed admit no solutions.
\begin{enumerate}
\item $G_1$: no solutions.
\item $G_2$: no solutions.
\item $G_3$: $\gamma(\alpha+\beta)=0$ and
\[\alpha v_2=0,\qquad\beta v_1=0,\qquad v_3\in\mathbb R.\]
Thus the $e_3$ component is always arbitrary; an arbitrary $e_1$ component is also allowed if $\beta=0$, and an arbitrary $e_2$ component if $\alpha=0$. In the abelian case, every $V$ is allowed.
\item $G_4$: $\alpha=\beta=\varepsilon$, $V=0$.
\item $G_5$: all admissible parameters, with
\[\begin{pmatrix}\alpha&\gamma\\\beta&\delta\end{pmatrix}\begin{pmatrix}v_1\\v_2\end{pmatrix}=0,\qquad v_3\in\mathbb R.\]
If the determinant is nonzero, $V=v_3e_3$. If it is zero, the horizontal component is arbitrary in the one-dimensional kernel, since $\alpha+\delta\ne0$.
\item $G_6$: exactly two branches:
\begin{enumerate}
\item $\alpha=\beta=0$, $\delta\ne0$, $\gamma\in\mathbb R$, with $V=v_1e_1+v_2e_2$ and $\gamma(2v_1+\delta)=0$;
\item $\alpha\ne0$, $\delta=-2\alpha$, $\gamma=2s\alpha$, $\beta=-s\alpha$, $s=\pm1$, with $V=0$.
\end{enumerate}
In (a), $v_1,v_2$ are arbitrary if $\gamma=0$; if $\gamma\ne0$, $V=-\delta e_1/2+te_2$, $t\in\mathbb R$.
\item $G_7$: exactly two branches:
\begin{enumerate}
\item $\alpha=0$, $\beta\ne0$, $\gamma=-\beta$, $\delta\ne0$, with $V=0$;
\item $\alpha=\beta=0$, $\delta\ne0$, $\gamma\in\mathbb R$, with $V=v_1e_1+t(e_2-e_3)$ and $\gamma(2t+\delta)=0$.
\end{enumerate}
In (b), $v_1$ is arbitrary. If $\gamma=0$, $t$ is arbitrary; if $\gamma\ne0$, $t=-\delta/2$.
\end{enumerate}
\end{theorem}
Every solution is steady. In $G_6$(a) and $G_7$(b), $S\ne0$ when $\gamma\ne0$; these are additional solutions with nonzero curvature tensors admitted by symmetrization. If a definition additionally requires $V\ne0$, simply remove the zero vector fields from the list. The theorem classifies all left-invariant $V$ and makes no claim to solve the partial differential equation for arbitrary, non-left-invariant vector fields.

\section{Componentwise Solution of the Symmetrized Affine Equation}
To avoid confusing the distribution $H$ with the deformation tensor, write $\mathcal H=\LB_Vg$, $S=(T+T^{\mathsf T})/2$, and
\[\mathcal E=\mathcal H+2S-2\lambda E,\qquad E=\diag(1,1,-1).\]
Since $\mathcal E$ is symmetric, only its six upper-triangular entries need vanish. Notice that $2S_{13}=U$ and $2S_{23}=W$, not $2U$ and $2W$. In every family, $\nB_{e_3}V\in H$. Consequently,
\[\mathcal E_{33}=0+0-2\lambda(-1)=2\lambda=0.\]
Thus $\lambda=0$ in all seven families, and every equation below is considered after this reduction.

\subsection{The family G1 and a contradiction from one diagonal entry}
A componentwise calculation gives
\begin{align*}
\nB_{e_1}V&=\alpha v_2e_1-\alpha v_1e_2,\\
\nB_{e_2}V&=\alpha v_3e_3,\\
\nB_{e_3}V&=(\alpha v_1-\beta v_2)e_1+(\beta v_1-\alpha v_2)e_2.
\end{align*}
In particular, $g(\nB_{e_2}V,e_2)=0$, so
\[\mathcal E_{22}=2K=-2(\alpha^2+\beta^2)=0.\]
This would require $\alpha=\beta=0$, contrary to $\alpha\ne0$. There is therefore no need to solve for $v_1,v_2,v_3$: this contradiction proves that the six equations have no common solution. Allowing the original tensor to be nonsymmetric does not remove the diagonal obstruction.

\subsection{The family G2 and a parameter contradiction after solving for V}
The connection gives
\[\nB_{e_1}V=-\gamma v_3e_3,\qquad
\nB_{e_2}V=\gamma v_2e_1-\gamma v_1e_2,\qquad
\nB_{e_3}V=-\alpha v_2e_1+\beta v_1e_2.\]
The remaining five equations are
\begin{align*}
(11):\;&K=0,&(12):\;&\gamma v_2=0,\\
(13):\;&-\alpha v_2+\gamma v_3=0,&(22):\;&K-\gamma v_1=0,\\
(23):\;&\beta v_1+\tfrac12\gamma(\beta-\alpha)=0.
\end{align*}
Since $\gamma\ne0$, (12) yields $v_2=0$, and (13) then gives $v_3=0$. Equations (11) and (22) imply $v_1=0$. Equation (23) now gives $\beta=\alpha$. Finally, (11) becomes $-\gamma^2-\beta^2=0$, which is impossible. One cannot assert $K<0$ at the outset, because $\alpha\beta$ may be negative; the sum-of-squares contradiction becomes valid only after the other equations establish $\alpha=\beta$.

\subsection{The family G3 and the explicit space of vector fields}
Here $A_1=A_2=0$ and $\nB_{e_3}V=-\alpha v_2e_1+\beta v_1e_2$. Therefore
\[\mathcal E=\begin{pmatrix}2K&0&-\alpha v_2\\0&2K&\beta v_1\\-\alpha v_2&\beta v_1&0\end{pmatrix}.\]
It vanishes precisely when $\gamma(\alpha+\beta)=0$, $\alpha v_2=0$, and $\beta v_1=0$. The vector fields can be fully expanded as follows, rather than left implicit in a kernel condition:
\begin{center}\begin{tabular}{@{}lll@{}}\toprule
Parameter branch&Additional condition&$V$\\\midrule
$\alpha\ne0,\ \beta\ne0$&$\gamma=0$ or $\alpha+\beta=0$&$te_3$\\
$\alpha=0,\ \beta\ne0$&$\gamma=0$&$te_2+ue_3$\\
$\alpha\ne0,\ \beta=0$&$\gamma=0$&$te_1+ue_3$\\
$\alpha=\beta=0$&$\gamma$ arbitrary&$te_1+ue_2+we_3$\\\bottomrule
\end{tabular}\end{center}
Here $t,u,w$ are arbitrary real numbers. Each row makes all six entries of $\mathcal E$ vanish by substitution. The final row includes the nonabelian case $\gamma\ne0$, not just the abelian algebra.

\subsection{The family G4 and absence of extra solutions after symmetrization}
The (33) equation already gives $\lambda=0$. The other equations are
\[K=0,\quad-v_2=0,\quad-\alpha v_2-v_3=0,\quad K+v_1=0,\quad\beta v_1+\frac{\alpha-\beta}{2}=0.\]
In order, these give $v_2=v_3=v_1=0$, and then $\alpha=\beta$. Hence
\[K=\beta(2\varepsilon-\beta)-1=-(\beta-\varepsilon)^2=0,\]
so $\alpha=\beta=\varepsilon$. Conversely, these parameters give $K=W=U=0$, and $V=0$ gives $\mathcal H=0$, so the full equation holds.

\subsection{The family G5 and complete parameterization of the two dimensional kernel}
Only $A_3$ may be nonzero, and $S=0$. Thus
\[M\begin{pmatrix}v_1\\v_2\end{pmatrix}=0,\qquad
M=\begin{pmatrix}\alpha&\gamma\\\beta&\delta\end{pmatrix},\qquad v_3\text{ arbitrary}.\]
The original constraint $\alpha\gamma+\beta\delta=0$ says that the two columns of $M$ are orthogonal in the usual Euclidean inner product on $\mathbb R^2$. If $\det M\ne0$, then $v_1=v_2=0$ and $V=te_3$.

If $\det M=0$, the columns are both linearly dependent and orthogonal. If two nonzero vectors are dependent, write $y=kx$; their inner product is $k\|x\|^2$, and orthogonality forces $k=0$. Thus at least one column vanishes. Using $\alpha+\delta\ne0$ leaves exactly two possibilities:
\begin{enumerate}
\item First column zero: $\alpha=\beta=0$, $\delta\ne0$, $\gamma$ arbitrary. The equations are $\gamma v_2=\delta v_2=0$, so $v_2=0$ and $V=te_1+ue_3$.
\item Second column zero: $\gamma=\delta=0$, $\alpha\ne0$, $\beta$ arbitrary. Then $\alpha v_1=\beta v_1=0$, so $v_1=0$ and $V=te_2+ue_3$.
\end{enumerate}
These cases cannot occur simultaneously because that would violate the trace condition. This gives explicit parameterizations of every vector field in the family, rather than an unresolved matrix kernel.

\subsection{The family G6 and the origin of the additional nonzero tensor solutions}
The covariant derivatives are
\[\nB_{e_1}V=\delta v_3e_3,\qquad
\nB_{e_2}V=\alpha v_2e_1-\alpha v_1e_2,\qquad
\nB_{e_3}V=-\gamma v_1e_2.\]
The complete independent conditions are
\begin{equation}
K=0,\quad\alpha v_2=0,\quad\delta v_3=0,\quad\alpha v_1=0,\quad
\gamma v_1=-\frac{\gamma(2\alpha+\delta)}2.
\label{eq:g6aff}
\end{equation}
\textbf{Branch 1: $\gamma=0$.} Then $K=-\alpha^2=0$ gives $\alpha=0$, the trace condition gives $\delta\ne0$, and the original constraint $\alpha\gamma-\beta\delta=0$ gives $\beta=0$. Equation \eqref{eq:g6aff} reduces to $\delta v_3=0$, so $v_3=0$ and $v_1,v_2$ are arbitrary.

\textbf{Branch 2: $\gamma\ne0$.} Dividing the last equation by $\gamma$ gives $v_1=-(2\alpha+\delta)/2$. Then $\alpha v_1=0$ gives $\alpha(2\alpha+\delta)=0$.
\begin{enumerate}
\item If $\alpha=0$, then $\delta\ne0$, and $K=-\beta\gamma/2=0$ gives $\beta=0$. Hence $v_1=-\delta/2$, $v_3=0$, and $v_2=t$ is free. We obtain $V=-\delta e_1/2+te_2$, which is nonzero.
\item If $\alpha\ne0$, then $\delta=-2\alpha\ne0$, so $v_1=v_2=v_3=0$. The original constraint becomes $\alpha(\gamma+2\beta)=0$, giving $\gamma=-2\beta$. Then $K=0$ yields $\beta^2=\alpha^2$, or $\beta=-s\alpha$, $\gamma=2s\alpha$, $s=\pm1$.
\end{enumerate}
\textbf{Full tensor substitution for the additional branch.} When $\alpha=\beta=0$, $\delta\ne0$, $\gamma\ne0$,
\[K=U=0,\quad W=-\gamma\delta/2,\qquad
S=\begin{pmatrix}0&0&0\\0&0&-\gamma\delta/4\\0&-\gamma\delta/4&0\end{pmatrix}.\]
For the above $V$, the only nonzero covariant derivative is $\nB_{e_3}V=\gamma\delta e_2/2$. Thus
\[\mathcal H=\begin{pmatrix}0&0&0\\0&0&\gamma\delta/2\\0&\gamma\delta/2&0\end{pmatrix}=-2S.\]
This explains why $W$ need not vanish: it is exactly canceled by the vector field term.

\subsection{The family G7 and a sum of squares obstruction}
The covariant derivatives are
\begin{align*}
\nB_{e_1}V&=-\alpha v_2e_1+\alpha v_1e_2+\beta v_3e_3,\\
\nB_{e_2}V&=-\beta v_2e_1+\beta v_1e_2+\delta v_3e_3,\\
\nB_{e_3}V&=(-\alpha v_1-\gamma v_2)e_1+(-\beta v_1-\delta v_2)e_2.
\end{align*}
The five equations are
\begin{align*}
K-\alpha v_2&=0,&\alpha v_1-\beta v_2&=0,\\
-\alpha v_1-\beta v_3-\gamma v_2+U&=0,&K+\beta v_1&=0,\\
-\beta v_1-\delta(v_2+v_3)+W&=0.
\end{align*}
\textbf{Exclude $\alpha\ne0$.} The original condition $\alpha\gamma=0$ implies $\gamma=0$, so $K=-\alpha^2-\beta^2/2<0$. Dividing successively by $\alpha$ in the first two equations gives
\[v_2=K/\alpha,\qquad v_1=\beta K/\alpha^2.\]
The fourth equation becomes $K(\alpha^2+\beta^2)/\alpha^2=0$. The first factor is strictly negative and the second strictly positive, so this is impossible. The entire $\alpha\ne0$ branch is excluded.

\textbf{Solve $\alpha=0$.} Then $\delta\ne0$. The first equation gives $K=0$, equivalently $\beta(\beta+\gamma)=0$, and $W=-K=0$. The second and fourth equations yield $\beta v_2=\beta v_1=0$. The fifth reduces to $\delta(v_2+v_3)=0$, so $v_3=-v_2$. The third becomes
\[(\beta-\gamma)v_2-\frac{\delta(\beta+\gamma)}2=0.\]
\begin{enumerate}
\item If $\beta\ne0$, then $K=0$ gives $\gamma=-\beta$. The equations $\beta v_1=\beta v_2=0$ give $v_1=v_2=0$, hence $v_3=0$. The third equation is automatically satisfied, with no additional condition.
\item If $\beta=0$, the third equation is $\gamma(2v_2+\delta)=0$, while $v_1$ is unrestricted. If $\gamma=0$, $v_2=t$ is arbitrary; if $\gamma\ne0$, $v_2=-\delta/2$ and $v_3=\delta/2$.
\end{enumerate}
\textbf{Full tensor substitution for the additional branch.} When $\alpha=\beta=0$, $\delta\ne0$, $\gamma\ne0$, we have $K=W=0$, $U=-\gamma\delta/2$. For $V=v_1e_1-\delta e_2/2+\delta e_3/2$,
\[\nB_{e_1}V=0,\qquad\nB_{e_2}V=\delta^2e_3/2,\qquad
\nB_{e_3}V=\gamma\delta e_1/2+\delta^2e_2/2.\]
Thus $\mathcal H_{13}=\gamma\delta/2=-2S_{13}$. In particular, the (23) component must be checked:
\[\mathcal H_{23}=g(\delta^2e_3/2,e_3)+g(e_2,\delta^2e_2/2)=-\delta^2/2+\delta^2/2=0.\]
The negative sign in the timelike direction is essential; omitting it would incorrectly exclude this branch. All remaining components vanish as well, so $\mathcal H+2S=0$.

\section{Contact Conditions and the Scope of the Results}
For a left-invariant one-form, $d\eta(e_i,e_j)=-\eta([e_i,e_j])$. Consequently,
\[(\eta\wedge d\eta)(e_1,e_2,e_3)=-C_{12}^{3}.\]
The conditions for $\eta$ to be a contact form are
\begin{center}\begin{tabular}{@{}ccccccc@{}}\toprule
$G_1$&$G_2$&$G_3$&$G_4$&$G_5$&$G_6$&$G_7$\\\midrule
$\beta\ne0$&$\beta\ne0$&$\gamma\ne0$&$\beta\ne2\varepsilon$&Never&$\beta\ne0$&$\beta\ne0$\\\bottomrule
\end{tabular}\end{center}
If only the contact-form condition is imposed, one simply intersects this table with the parameter conditions in the main theorems. For example, within the self-adjoint algebraic classification, the contact-form subfamilies are exactly $G_3$ with $\gamma\ne0$, $\alpha+\beta=0$; $G_4$ with $\alpha=\beta=\varepsilon$; $G_6$(b); and $G_7$ with $\alpha=0$, $\gamma=-\beta$, $\delta\ne0$, $\beta\ne0$.

A contact metric structure requires, in addition, that the specified $\xi$ be the Reeb vector field and that $d\eta$ satisfy the chosen compatibility normalization with $\phi,g$. In this basis, $\iota_\xi d\eta=0$ is equivalent to $C_{13}^{3}=C_{23}^{3}=0$. The $G_4,G_6,G_7$ cases in the intersection above therefore fail the Reeb condition for the specified $\xi=e_3$. Under the normalization $d\eta(X,Y)=g(X,\phi Y)$, the only listed solutions satisfying the contact metric condition are those of $G_3$ with $\gamma=-1$, $\alpha+\beta=0$. With the normalization containing a factor of $2$, $\gamma$ becomes $-2$. These are additional restrictions, not default assumptions of the preceding classifications.

\section{Discussion}
The computations show that the non-self-adjoint part of the original star-Ricci endomorphism genuinely enters the derivation condition; its diagonal block alone is insufficient. In the symmetrized affine problem, on the other hand, mixed components can be canceled by $\LB_Vg$, permitting nonzero $S$. These phenomena reflect the different structures of the algebraic derivation equation and the affine tensor equation.

For the fixed structure, the classification exhausts all admissible real parameters and all left-invariant vector fields. Changing the two-dimensional distribution, or allowing vector field coefficients to vary from point to point, produces different equation systems. In particular, the problem for arbitrary vector fields contains derivative terms $e_i(v_j)$, and the finite-dimensional elimination used here does not directly prove a classification in that setting.

\textbf{Reproducibility of the computations.} The appendix gives the complete connection and curvature matrices. Every classification proof is carried out in the main text. Symbolic computation is used only to check matrix multiplication and substitution, not to replace the necessity arguments. The Python program accompanying the source files can recompute all nine tensor components and every retained branch.

\appendix
\section{Bott Connection Matrices and Reproducible Data}
All matrices below use columns to represent images. Together with the stated Lie brackets and \eqref{eq:curvature}, these matrices generate all 27 curvature vectors.

\subsection{The family $G_1$}
\[A_1=\begin{pmatrix}0&\alpha&0\\-\alpha&0&0\\0&0&0\end{pmatrix}.\]
\[A_2=\begin{pmatrix}0&0&0\\0&0&0\\0&0&\alpha\end{pmatrix}.\]
\[A_3=\begin{pmatrix}\alpha&-\beta&0\\\beta&-\alpha&0\\0&0&0\end{pmatrix}.\]
The three independent curvature matrices are
\[R_{12}=\begin{pmatrix}\alpha\beta&-\alpha^2-\beta^2&0\\\alpha^2+\beta^2&-\alpha\beta&0\\0&0&0\end{pmatrix}.\]
\[R_{13}=\begin{pmatrix}0&-\alpha^2&0\\-3\alpha^2&0&0\\0&0&\alpha\beta\end{pmatrix}.\]
\[R_{23}=\begin{pmatrix}-\alpha^2&0&0\\0&\alpha^2&0\\0&0&-\alpha^2\end{pmatrix}.\]

\subsection{The family $G_2$}
\[A_1=\begin{pmatrix}0&0&0\\0&0&0\\0&0&-\gamma\end{pmatrix}.\]
\[A_2=\begin{pmatrix}0&\gamma&0\\-\gamma&0&0\\0&0&0\end{pmatrix}.\]
\[A_3=\begin{pmatrix}0&-\alpha&0\\\beta&0&0\\0&0&0\end{pmatrix}.\]
The three independent curvature matrices are
\[R_{12}=\begin{pmatrix}0&-\alpha\beta-\gamma^2&0\\\beta^2+\gamma^2&0&0\\0&0&0\end{pmatrix}.\]
\[R_{13}=\begin{pmatrix}0&-\gamma(\alpha-\beta)&0\\0&0&0\\0&0&0\end{pmatrix}.\]
\[R_{23}=\begin{pmatrix}-\gamma(\alpha-\beta)&0&0\\0&\gamma(\alpha-\beta)&0\\0&0&\alpha\gamma\end{pmatrix}.\]

\subsection{The family $G_3$}
\[A_1=\begin{pmatrix}0&0&0\\0&0&0\\0&0&0\end{pmatrix}.\]
\[A_2=\begin{pmatrix}0&0&0\\0&0&0\\0&0&0\end{pmatrix}.\]
\[A_3=\begin{pmatrix}0&-\alpha&0\\\beta&0&0\\0&0&0\end{pmatrix}.\]
The three independent curvature matrices are
\[R_{12}=\begin{pmatrix}0&-\alpha\gamma&0\\\beta\gamma&0&0\\0&0&0\end{pmatrix}.\]
\[R_{13}=\begin{pmatrix}0&0&0\\0&0&0\\0&0&0\end{pmatrix}.\]
\[R_{23}=\begin{pmatrix}0&0&0\\0&0&0\\0&0&0\end{pmatrix}.\]

\subsection{The family $G_4$}
\[A_1=\begin{pmatrix}0&0&0\\0&0&0\\0&0&1\end{pmatrix}.\]
\[A_2=\begin{pmatrix}0&-1&0\\1&0&0\\0&0&0\end{pmatrix}.\]
\[A_3=\begin{pmatrix}0&-\alpha&0\\\beta&0&0\\0&0&0\end{pmatrix}.\]
The three independent curvature matrices are
\[R_{12}=\begin{pmatrix}0&-\alpha\beta+2\alpha\varepsilon-1&0\\\beta^2-2\beta\varepsilon+1&0&0\\0&0&0\end{pmatrix}.\]
\[R_{13}=\begin{pmatrix}0&\alpha-\beta&0\\0&0&0\\0&0&0\end{pmatrix}.\]
\[R_{23}=\begin{pmatrix}\alpha-\beta&0&0\\0&-\alpha+\beta&0\\0&0&-\alpha\end{pmatrix}.\]

\subsection{The family $G_5$}
\[A_1=\begin{pmatrix}0&0&0\\0&0&0\\0&0&0\end{pmatrix}.\]
\[A_2=\begin{pmatrix}0&0&0\\0&0&0\\0&0&0\end{pmatrix}.\]
\[A_3=\begin{pmatrix}-\alpha&-\gamma&0\\-\beta&-\delta&0\\0&0&0\end{pmatrix}.\]
The three independent curvature matrices are
\[R_{12}=\begin{pmatrix}0&0&0\\0&0&0\\0&0&0\end{pmatrix}.\]
\[R_{13}=\begin{pmatrix}0&0&0\\0&0&0\\0&0&0\end{pmatrix}.\]
\[R_{23}=\begin{pmatrix}0&0&0\\0&0&0\\0&0&0\end{pmatrix}.\]

\subsection{The family $G_6$}
\[A_1=\begin{pmatrix}0&0&0\\0&0&0\\0&0&\delta\end{pmatrix}.\]
\[A_2=\begin{pmatrix}0&\alpha&0\\-\alpha&0&0\\0&0&0\end{pmatrix}.\]
\[A_3=\begin{pmatrix}0&0&0\\-\gamma&0&0\\0&0&0\end{pmatrix}.\]
The three independent curvature matrices are
\[R_{12}=\begin{pmatrix}0&-\alpha^2&0\\\alpha^2+\beta\gamma&0&0\\0&0&0\end{pmatrix}.\]
\[R_{13}=\begin{pmatrix}0&-\alpha\gamma&0\\\gamma(\alpha+\delta)&0&0\\0&0&0\end{pmatrix}.\]
\[R_{23}=\begin{pmatrix}-\alpha\gamma&0&0\\0&\alpha\gamma&0\\0&0&0\end{pmatrix}.\]

\subsection{The family $G_7$}
\[A_1=\begin{pmatrix}0&-\alpha&0\\\alpha&0&0\\0&0&\beta\end{pmatrix}.\]
\[A_2=\begin{pmatrix}0&-\beta&0\\\beta&0&0\\0&0&\delta\end{pmatrix}.\]
\[A_3=\begin{pmatrix}-\alpha&-\gamma&0\\-\beta&-\delta&0\\0&0&0\end{pmatrix}.\]
The three independent curvature matrices are
\[R_{12}=\begin{pmatrix}-\alpha\beta&-\alpha^2-\beta^2-\beta\gamma&0\\\alpha^2&-\beta\delta&0\\0&0&\beta(\alpha+\delta)\end{pmatrix}.\]
\[R_{13}=\begin{pmatrix}\alpha(2\beta+\gamma)&\alpha\delta+\beta^2+\beta\gamma&0\\-\alpha(2\alpha-\delta)&-\alpha\beta-\alpha\gamma+\beta\delta&0\\0&0&-\beta(\alpha+\delta)\end{pmatrix}.\]
\[R_{23}=\begin{pmatrix}\alpha\delta+\beta^2+\beta\gamma&-\alpha\beta+\alpha\gamma+2\beta\delta+\delta\gamma&0\\-\alpha\beta-\alpha\gamma+\beta\delta&-\beta^2-\beta\gamma+\delta^2&0\\0&0&-\beta\gamma-\delta^2\end{pmatrix}.\]

\end{document}